\documentclass[12pt]{amsart}
\usepackage[utf8]{inputenc}
\usepackage[margin=1in]{geometry}
\usepackage{amsmath,amssymb,amsthm}
\usepackage{graphicx}
\usepackage{color}
\usepackage{hyperref} 
\hypersetup{
    colorlinks=true,
    linkcolor=blue,
    filecolor=blue,      
    urlcolor=blue,
    citecolor=blue
}
\newtheorem{theorem}{Theorem}[section]
\newtheorem{proposition}[theorem]{Proposition}
\newtheorem{corollary}[theorem]{Corollary}

\theoremstyle{definition}
\newtheorem{definition}[theorem]{Definition}

\newtheorem{question}[theorem]{Question}

\theoremstyle{remark}
\newtheorem{remark}{Remark}

\numberwithin{equation}{section}

\newcommand{\R}{\mathbb{R}}

\newcommand{\C}{\mathbb{C}}

\DeclareMathOperator{\im}{Im}

\newcommand{\abs}[1]{\lvert#1\rvert}

\title{On Zero-Sector Reducing Operators}

\author[D.~A.~Cardon]{David~A.~Cardon}
\address{Department of Mathematics, Brigham Young University, Provo, UT 84602, USA}
\email{cardon@math.byu.edu}
\author[T. Forg\'acs]{Tam\'as Forg\'acs}
\address{Department of Mathematics, M/S PB108, Fresno, CA 93740-8001, USA}
\email{tforgacs@csufresno.edu}
\author[A.~Piotrowski]{Andrzej~Piotrowski}
\address{Department of Natural Sciences, M/S SOB1, University of Alaska Southeast, Juneau, AK 99801, USA}
\email{apiotrowski@alaska.edu}
\author[E.~Sorensen]{Evan~Sorensen}
\address{Department of Mathematics, Brigham Young University, Provo, UT 84602, USA}
\email{esorensencapps@gmail.com}
\author[J.~C.~White]{Jason~C.~White}
\address{Department of Mathematics, Brigham Young University, Provo, UT 84602, USA}
\email{white.jason.c@gmail.com}
\begin{document}

%%%%%%%%%%%%%%%%%%%%%%%%%%%%%%%%%%%%%%%%%%%%%%%%%%%%%%%%%%%%%%%%%%%%%%%%%%%%%%%%%%%%%%%%%%%%%%%%%%%%%%%%%%%%%%%%%%%%%%%%%%%%%%%%%%%%%%%%%%%%%%%%%%%%%%%%%%%%%%%%

\begin{abstract}
We prove a Jensen-disc type theorem for polynomials $p\in\R[z]$ having all their zeros in a sector of the complex plane. This result is then used to prove the existence of a collection of linear operators $T\colon\R[z]\to\R[z]$ which map polynomials with their zeros in a closed convex sector $\abs{\arg z} \leq \theta<\pi/2$ to polynomials with zeros in a smaller sector $\abs{\arg z} \leq \gamma<\theta$. We, therefore, provide the first example of a \textit{zero-sector reducing operator}.
\end{abstract}

%%%%%%%%%%%%%%%%%%%%%%%%%%%%%%%%%%%%%%%%%%%%%%%%%%%%%%%%%%%%%%%%%%%%%%%%%%%%%%%%%%%%%%%%%%%%%%%%%%%%%%%%%%%%%%%%%%%%%%%%%%%%%%%%%%%%%%%%%%%%%%%%%%%%%%%%%%%%%%%%

\maketitle
%\tableofcontents

%%%%%%%%%%%%%%%%%%%%%%%%%%%%%%%%%%%%%%%%%%%%%%%%%%%%%%%%%%%%%%%%%%%%%%%%%%%%%%%%%%%%%%%%%%%%%%%%%%%%%%%%%%%%%%%%%%%%%%%%%%%%%%%%%%%%%%%%%%%%%%%%%%%%%%%%%%%%%%%%

\section{Introduction}

For any polynomial $p\in\C[z]$, the classical Gauss-Lucas Theorem states that the zeros of the derivative $p'$ lie inside the closed convex hull of the zeros of $p$. Jensen proved a more precise result in the case where the polynomial $p$ has real coefficients. Jensen's theorem states that all of the non-real zeros of the derivative of a polynomial $p\in\R[x]$ must lie in at least one of the Jensen discs for $p$, where a \textit{Jensen disc} for $p$ is a closed disc whose diameter connects a conjugate pair of non-real zeros of $p$~\cite{J}. 

Either of the results just mentioned demonstrate that the differentiation operator on $\R[x]$ (or $\C[x]$, in the case of the Gauss-Lucas Theorem) maps polynomials with zeros in a strip 
\[
\sigma(A) = \{z \colon \abs{\im z} \leq A\}
\]
to polynomials with zeros in that same strip. Thus, differentiation is an example of a \textit{zero-strip preserving operator}. Bleecker and Csordas use a result of de Bruijn to demonstrate that some differential operators such as $\exp(-\alpha^2 D^2/2)$, 
where $\alpha>0$, map polynomials with zeros in the strip $\sigma(A)$ to a strictly smaller strip $\sigma(A')$, where $A' =\sqrt{\max\{A^2-\alpha^2, 0\}}$ (see \cite[Theorem~3.2]{BC}). Such operators are called \textit{complex zero-strip decreasing operators} and they have been studied in detail by the first author~\cite{Cardon2015}. 

For a closed convex sector 
\[
S(\theta) = \{z \colon \abs{\arg z} \leq \theta \text{ or } z=0 \}\qquad (0\leq \theta<\pi/2),
\]
there are known results which demonstrate the existence of \textit{zero-sector preserving operators} (see, for example,~\cite[Chapter 4]{C}). One of the main results of this paper is to demonstrate the existence of a collection of zero-sector \textit{reducing} operators (Theorem~\ref{ZSRO}). We do so by proving a Jensen disc-type theorem for polynomials with their zeros in a sector (Theorem~\ref{JSD}). 

In the extreme case, the strip degenerates to the real line and the sector degenerates to the non-negative real axis. In~\cite{PS}, P\'olya and Schur characterized all linear operators on $\R[z]$ of the form $T[z^n] = \gamma_n z^n$ which preserve the location of zeros on these limiting sets. They termed the sequence $\{\gamma_k\}_{k=0}^{\infty}$ corresponding to an operator $T$ which maps polynomials with only real zeros to polynomials with only real zeros a \textit{multiplier sequence of the first kind}. Similarly, they termed the sequence $\{\gamma_k\}_{k=0}^{\infty}$ corresponding to an operator $T$ which maps polynomials with only positive real zeros to polynomials with only real zeros a \textit{multiplier sequence of the second kind}. A multiplier sequence of the second kind can be thought of as an operator which maps polynomials with zeros in the sector $S(0)$ to polynomials with zeros in the double sector 
\[
\pm S(0) = \{z \colon z\in S(0) \text{ or } -z\in S(0)\}.
\]
Our results will yield new proofs of some of the classical results. In particular, we will provide a new proof of a result due to Laguerre~\cite{L} which states that the sequence $\{\cos(\lambda+k \theta)\}_{k=0}^{\infty}$, where $\lambda$ and $\theta$ are real, is a multiplier sequence of the second kind (Corollary \ref{Lms2}).

\begin{remark}
Throughout this paper, we will continue to use the notation $\sigma(A)$, $S(\theta)$, and $\pm S(\theta)$ to denote the strip, sector, and double sector, respectively, as defined in this introduction. 
\end{remark}

%%%%%%%%%%%%%%%%%%%%%%%%%%%%%%%%%%%%%%%%%%%%%%%%%%%%%%%%%%%%%%%%%%%%%%%%%%%%%%%%%%%%%%%%%%%%%%%%%%%%%%%%%%%%%%%%%%%%%%%%%%%%%%%%%%%%%%%%%%%%%%%%%%%%%%%%%%%%%%%%

\section{Some Zero-Sector Reducing Operators}

We next extend the notion of a Jensen disk from the setting of horizontal strips containing the roots of real polynomials to the case in which the roots of real polynomials belong to a sector.
\begin{definition}
Suppose $a$ and $b$ are positive real numbers and $a+ib$ is a zero of $p\in\R[z]$ with $\arg (a+i b)=\theta$. For $0\leq\alpha\leq\pi$ and $\abs{\sec \alpha} < \sec \theta$, we define the \textbf{Jensen sector-disc} corresponding to $a+ib$ and $\alpha$ as the closed disc $\Delta(a,b;\alpha)$ with center $c = (\cos \alpha )(a^2+b^2)/a  = a \cos \alpha \sec^2 \theta$, and radius $r = \sqrt{c^2 - a^2-b^2}.$ In the case $\abs{\sec \alpha}$ is not less than $\sec \theta$, we define $\Delta(a,b; \alpha)=\emptyset$. The Jensen sector-disc is depicted in Figure~\ref{figure:Jensen-sector-disc}.
\end{definition}

\begin{figure}[hbt!] 
\caption{The Jensen sector-disc $\Delta(a,b;\alpha)$}
\label{figure:Jensen-sector-disc}
\includegraphics[width = 4 in]{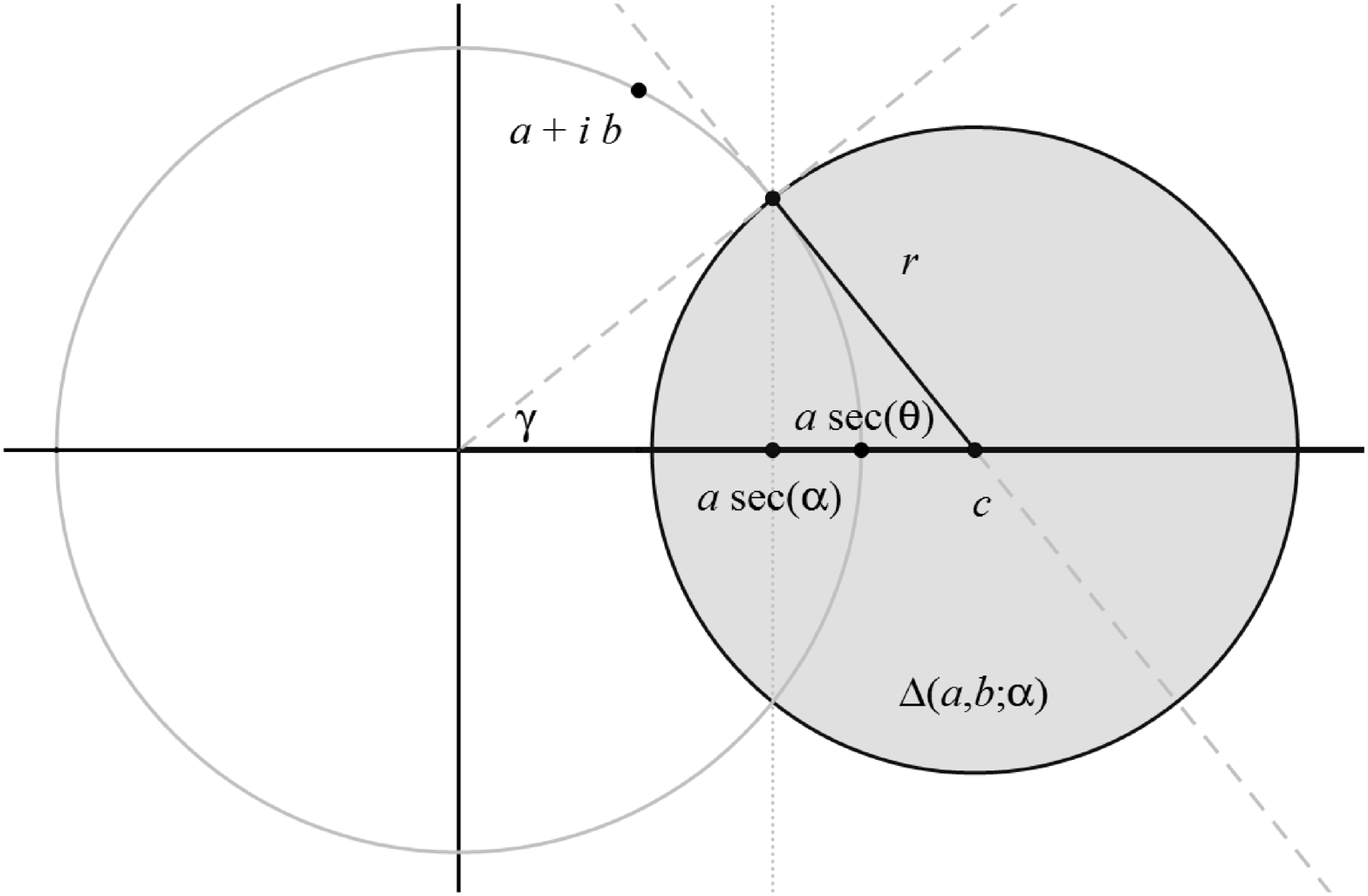}
\end{figure}

\begin{remark} \label{JSDrem}
Geometrically, the Jensen sector-disc $\Delta(a,b; \alpha)$ is the disc which is tangent to the two rays $\arg z=\pm\gamma$, where $\cos\gamma=\cos\theta\sec\alpha$, with the points of tangency lying on the circle $|z|=|a+ib|$ (see  Figure~\ref{figure:Jensen-sector-disc}).
\end{remark}

\begin{theorem}\label{JSD}
Suppose
\[
p(z) = \prod_{k=1}^{m}(z-x_k) \prod_{k=1}^{n} [(z-a_k)^2+b_k^2],
\]
where $x_k \geq 0$, $a_k>0$, and $b_k>0$ for all $k$. For fixed real numbers $\alpha, \beta,$ and $\lambda$ with $0\leq\alpha\leq\pi$,
define 
\[
f(z) = {e^{i \lambda}p(e^{i\alpha} z)+e^{i \beta}p(e^{-i\alpha} z)}.
\]
Then every non-real zero of $f$ must lie in at least one Jensen sector-disc $\Delta(a_k, b_k; \alpha)$. 
\end{theorem}

\begin{proof}
If $\alpha  = 0$, then $f(z) = \left(e^{i \lambda}+e^{i \beta}\right)p(z)$ and the theorem follows from the fact that $a\pm ib$ is contained in the Jensen sector-disc $\Delta(a, b, 0)$. Similarly, if $\alpha=\pi$, then $f(z) = \left(e^{i \lambda}+e^{i \beta}\right)p(-z)$ and the theorem follows from the fact that $a\pm ib$ is contained in the Jensen sector-disc $\Delta(a, b, \pi)$.
 
We will now prove the theorem in the case where $0<\alpha<\pi$. Note that if $f(z_0) = 0$, then $\abs{ p(e^{-i\alpha} z_0) } = \abs{ p(e^{i\alpha} z_0) }$. Suppose $z=x+iy$ lies in the upper half-plane, so that $y>0$. Since $y$, $x_k$, and $\sin \alpha$ are all positive, we have   
\[
\sqrt{(x-x_k\cos \alpha)^2+(y-x_k \sin \alpha)^2} < \sqrt{(x-x_k\cos \alpha)^2+(y+x_k \sin \alpha)^2}.
\] 
Thus, for any $z$ in the upper half-plane, the linear factors of $p(e^{\pm i \alpha}z)$ satisfy
\[
\abs{ e^{-i\alpha} z - x_k } <  \abs{ e^{i\alpha} z - x_k }.
\]

For the quadratic factors, we claim that $\abs{ (e^{-i \alpha} z - a_k)^2+b_k^2 } <  \abs{ (e^{i \alpha} z - a_k)^2+b_k^2 }$ if and only if $z\notin \Delta(a_k, b_k; \alpha)$. Indeed, a calculation shows that 
\[
\abs{ (e^{i\alpha} z - a_k)^2+b_k^2 }^2 - \abs{ (e^{-i\alpha} z - a_k)^2+b_k^2 }^2 = 8y \sin \alpha [a_k(a_k^2+b_k^2+x^2+y^2) - 2 x (a_k^2+b_k^2)\cos\alpha ].
\]
Since $a_k$, $y$, and $\sin \alpha$ are all positive, we see that $\abs{(e^{-i \alpha} z - a_k)^2+b_k^2 } < \abs{ (e^{i \alpha} z - a_k)^2+b_k^2 }$ is equivalent to
\[
\cos^2(\alpha) \left(\tfrac{a_k^2+b_k^2}{a_k}\right)^2-a_k^2-b_k^2<\left(x- \cos(\alpha) \tfrac{a_k^2+b_k^2}{a_k}\right)^2 + y^2.
\]
Comparing this with the center and radius of the Jensen sector-disc gives the desired result for points in the upper half-plane.

Finally, a calculation shows that $f(z) = 0$ implies $f(\overline{z}) = 0$ so that the non-real zeros of $f$ come in conjugate pairs, which completes the proof.
\end{proof}

Some remarks are in order:
\begin{enumerate}
\item Theorem~\ref{JSD} remains valid for any real $\alpha$, as long as the appropriate reference angle is used in the Jensen sector-disc. 
\item The polynomial $p$ in Theorem~\ref{JSD} need not have any non-real zeros. In this case, $n=0$ and all of the zeros of $f$ are real.
\item The theorem is sharp in the sense that for quadratic functions $p(z) = (z-a)^2+b^2$, the non-real zeros of the corresponding function $f$ lie on the boundary of $\Delta(a,b;\alpha)$.

\end{enumerate}

As a corollary, we obtain a new proof of a classical result due to Laguerre.

\begin{corollary} \label{Lms2} \emph{(C.f.,~\cite{L})}
For any fixed real numbers $\lambda$ and $\theta$, the sequence $\{\cos(\lambda+ k\theta)\}_{k=0}^{\infty}$ is a multiplier sequence of the second kind.
\end{corollary}

\begin{proof}
By employing the transformation $z\to-z$, if necessary, we may assume that all the zeros of 
\[
p(z) = \sum_{k=0}^{n} c_k z^k \qquad (c_k\in\R)
\]
are positive. By Theorem~\ref{JSD}, the zeros of 
\[
\frac{1}{2}\left[e^{i\lambda} p(e^{-i\theta} z)+e^{-i\lambda} p(e^{i\theta} z) \right] = \sum_{k=0}^{n} \cos(\lambda+k\theta) c_k z^k 
\]
are all real.
\end{proof}

\begin{remark} Laguerre's original proof (see \cite[p.\,204-206]{L} uses the following result due to Hermite: Let $F(x)=F_1(x)+i F_2(x)$, where $F_1,F_2$ are real polynomials. If the imaginary parts of the solutions of $F(x)=0$ are all of the same sign, then the function $\alpha F_1(x)+\beta F_2(x)$ has only real zeros for any choices of real numbers $\alpha, \beta$. \\
If the zeros of a polynomial $p$ are all of the same sign, then $p(xe^{i \omega})$ has zeros whose imaginary parts all have the same sign, and hence by Hermite's theorem (with $\alpha=\cos \lambda, \beta=-\sin \lambda$) the result follows.
\end{remark}
We now present another consequence of Theorem~\ref{JSD} which will be used to prove the existence of an operator which reduces the size of the zero-containing sector for a real polynomial. 
\begin{corollary} \label{cos(ak)}
Suppose all of the zeros of the polynomial
\[
p(z) = \sum_{k=0}^{n} c_k z^k
\]
lie in the sector $S(\theta)$, where $0\leq \theta<\pi/2$. If $0<\alpha<\pi/2$ and if $N$ satisfies $\alpha n/N<\pi/2$, then all of the zeros of the polynomial
\[
q(z) = \sum_{k=0}^{n} c_k \cos(\alpha k/N) z^k
\] 
lie in the sector $S(\gamma)$, where $\gamma = \arccos( \min\{1,  \cos \theta \sec (\alpha/N)\}).$
\end{corollary}

\begin{proof}
Since
\[
q(z) = \frac{1}{2}\left[p(e^{-i\alpha/N} z)+ p(e^{i\alpha/N} z) \right], 
\]
Theorem \ref{JSD} guarantees that all of the non-real zeros of $q$ lie in the union of the Jensen-sector discs associated with any non-real zeros of $p$. In light of Remark \ref{JSDrem}, all of the non-real zeros of $q$ must lie in the sector $S(\gamma)$. 

To finish the proof, we need to show that none of the real zeros of $q$ is negative. Since the zeros of $p$ lie in the right half-plane, its coefficients must alternate in sign (see, for example,~\cite[Prop.~11.4.2]{RS}). Since all of the quantities   
\[
\cos(\alpha k/N), \qquad k=0, 1, 2, \dots, n
\]
are positive the coefficients of $q$ must also alternate in sign. Therefore, none of the real zeros of $q$ can be negative. 
\end{proof}

We now establish the existence of a collection of \textit{zero-sector reducing operators}.

\begin{theorem}\label{ZSRO}
Let
\[
p(z) = c_{m+2n}\prod_{k=1}^{m}(z-x_k) \prod_{k=1}^{n} [(z-a_k)^2+b_k^2] = \sum_{k=0}^{m+2n} c_k z^k,
\]
where $x_k \geq 0$, $a_k>0$, and $b_k>0$ for all $k$. If all the zeros of $p$ lie in the closed sector $S(\theta)$, where $0\leq \theta<\pi/2$, and if $\alpha\in\R$, then all of the zeros of
\[
T[p(z)] =  \sum_{k=0}^{m+2n} \exp\left(\frac{-\alpha^2 k^2}{2}\right) c_k z^k
\]
lie in the sector $S(\gamma)$, where $\gamma = \arccos\left(\min\left\{1, e^{\alpha^2/2} \cos\theta\right\}\right).$
\end{theorem}

\begin{proof}
We will obtain the exponential sequence as a limit of a cosine sequence via the limit
 
\[
\lim_{N\to \infty} [\cos (\alpha /N)]^{N^2}  = e^{-\alpha^2/2}.
\]
For all $N$ sufficiently large, we may apply Corollary~\ref{cos(ak)} a total of $N^2$ times, and each application reduces the sector half-angle by a factor of $\sec(\alpha/N)$. This results in the polynomial 
\[
p_N(z) = \sum_{k=0}^{m+2n} c_k \left[\cos(\alpha k/N)\right]^{N^2} z^k,
\]
whose zeros must lie in the sector $S(\gamma_N)$, where $\gamma_N$ satisfies
\[
\gamma_N = \arccos(\min\{1, \cos \theta \left[\sec(\alpha/N)\right]^{N^2-1}\}).
\] 
Taking the limit as $N\to \infty$, we see that the zeros of 
\[
T[p(z)] =  \sum_{k=0}^{m+2n}  \exp\left(\frac{-\alpha^2 k^2}{2}\right) c_k z^k
\]
must lie in the sector $S(\gamma)$, where $\gamma$ satisfies
\[
\gamma = \arccos(\min\{1, e^{\alpha^2/2}\cos \theta \}) .
\]
The result follows.
\end{proof}

We note here that the result in Theorem~\ref{ZSRO} is sharp, with quadratic polynomials having their zeros on the boundary of the reduced sector. Furthermore, the example $p(z) = z^2+1$ demonstrates that the theorem cannot be extended to include the case where the sector is the closed half-plane ($\theta = \pi/2$). 

To close this section, we extend the foregoing results to certain transcendental entire functions. The result follows directly from Hurwitz' theorem (\cite[Theorem
~6.2.6]{MH}).
\begin{corollary}
If 
\[
f(z) = \sum_{k=0}^{\infty} c_k z^k
\]
is the uniform limit on compact subsets of $\C$ of a sequence of real polynomials $\{p_n\}$ each of which has zeros in the sector $S(\theta)$ for some $\theta<\pi/2$, then
all of the zeros of 
\[
T[f(z)] =  \sum_{k=0}^{\infty}  \exp\left(\frac{-\alpha^2 k^2}{2}\right) c_k z^k
\]  
lie in the sector $S(\gamma)$, where $\gamma = \arccos\left(\min\left\{1, e^{\alpha^2/2} \cos\theta\right\}\right).$
\end{corollary}

\section{Guiding Principles}
We now give some insight as to why we chose to investigate the operator $T$ appearing in Theorem $\ref{ZSRO}$. The heuristic property that guided our investigation is as follows. Suppose 
\[
p(z) = \sum_{k=0}^{n} a_k z^k = a_n\prod_{k=1}^{n}(z-z_k)
\]
has all of its zeros in the sector $S(\theta)$ with $\theta<\pi/2$. Then the zeros of $p(e^z)$ are given by $\log(z_k) = \ln\abs{z_k} + i\arg(z_k)$. The zeros corresponding to principle values of the logarithm all lie in the principle strip $\sigma(\theta)$ and the remaining zeros are identical copies located in strips at a vertical distance of integer multiples of $2\pi k$. The operator $\exp(-\alpha^2 D^2/2)$ acts on $e^{kz}$ as follows:
\[
e^{-\alpha^2 D^2/2} e^{k z} = \sum_{j=0}^{\infty} \frac{(-\alpha^2/2)^j}{j!} \frac{d^{2j}}{dz^{2j}} e^{kz} = e^{-\alpha^2 k^2/2}e^{kz}.
\]
Since the the operator $\exp(-\alpha^2 D^2/2)$ is a complex zero strip decreasing operator for all $\alpha>0$, it stands to reason that the zeros of 
\[
e^{-\alpha^2 D^2/2} p(e^z) = \sum_{k=0}^{n} a_k e^{-\alpha^2 k^2/2}e^{kz}
\] 
would have all its zeros in a collection of narrower %smaller 
strips. Thus, all of the zeros of $T[p]$ should have all of its zeros in a smaller sector $S(\gamma)$, where $\gamma<\theta$. The main obstacle in this argument is that the zeros of $p(e^z)$ do not lie in a \textit{single} strip in the complex plane. Thus, the complex zero strip decreasing operator $\exp(-\alpha^2 D^2/2)$ is not, {\it à priori}, guaranteed to reduce the size of each of the strips. However, with Theorem \ref{ZSRO} at our disposal, we obtain this fact as a corollary. 

\begin{corollary} \label{PeriodStrip}
For any $\alpha>0$, the infinite order differential operator $\exp(-\alpha^2 D^2/2)$  
is a periodic-strip decreasing operator on the space of exponential polynomials. That is to say, if all of the principle zeros of the exponential polynomial
\[
p(z) = \sum_{k=0}^{n} c_k e^{kz} \qquad (c_k\in\R)
\]
lie in the principle strip $\sigma(A)$, then the principle zeros of
\[
\exp(-\alpha^2 D^2/2) p(z) = \sum_{k=0}^{n} c_k e^{-\alpha^2 k^2/2}e^{kz}
\]
all lie in a smaller strip $\sigma(A^*)$, where $A^*=\arccos(\min\{1, e^{\alpha^2/2}\cos A \})$. 
\end{corollary}

We conclude this section by noting that the reduction of the strip size guaranteed by Corollary \ref{PeriodStrip} is somewhat less than what one would obtain from naively applying the theorem of Bleecker and Csordas alluded to in the introduction (see~\cite[Theorem~3.2]{BC}) to the exponential polynomial $p(e^z)$. In this case, one would obtain the strip $\sigma(A')$ where $A' = \sqrt{\max\{A^2-\alpha^2, 0\}}$. This demonstrates that zero-strip reducing properties on $\R[z]$ cannot be extended verbatim to principle zero-strip reducing properties on $\R[e^z]$. 

%%%%%%%%%%%%%%%%%%%%%%%%%%%%%%%%%%%%%%%%%%%%%%%%%%%%%%%%%%%%%%%%%%%%%%%%%%%%%%%%%%%%%%%%%%%%%%%%%%%%%%%%%%%%%%%%%%%%%%%%%%%%%%%%%%%%%%%%%%%%%%%%%%%%%%%%%%%%%%%%

\section{Necessary Conditions}
In this section, we determine some necessary conditions that a sequence $\{\gamma_k\}_{k=0}^{\infty}$ must satisfy in order for the corresponding linear operator defined by $T[z^n] = \gamma_n z^n$ to be a zero-sector reducing operator. 

%%%%%%%%%%%%%%%%%%%%%%%%%%%%%%%%%%%%%%%%%%%%%%%%%%%%%%%%%%%%%%%%%%%%%%%%%%%%%%%%%%%%%%%%%%%%%%%%%%%%%%%%%%%%%%%%%%%%%%%%%%%%%%%%%%%%%%%%%%%%%%%%%%%%%%%%%%%%%%%%

\subsection{Connection to Zero-Sector Preservers and Multiplier Sequences}

In order for an operator on $\R[z]$ to be a zero-sector reducing operator, it must be a zero-sector \textit{preserving} operator. The converse is not necessarily true. For example, the differentiation operator is a zero-sector preserver (a consequence of the Gauss-Lucas Theorem), but it will not reduce the size of the zero-sector when it acts on a polynomial with multiple zeros on the boundary of the sector. A good overview of zero-sector preservers is given in~\cite[Chapter 4]{C}. 

In order to connect zero-sector preserving operators with multiplier sequences, we recall the algebraic characterization of multiplier sequences of the first kind.
\begin{theorem} \emph{(Algebraic Characterization of Multiplier Sequences of the First Kind, ~\cite{PS})}
A sequence $\{\gamma_k\}_{k=0}^{\infty}$ of non-negative real numbers is a multiplier sequence of the first kind if and only if 
\[
T[(1+z)^n] = \sum_{k=0}^{n} \binom{n}{k} \gamma_k z^k
\]
has only real zeros for all positive integers $n$.
\end{theorem}
This is a complete characterization as a multiplier sequence of the first kind must either have constant sign, or must alternate in sign (see~\cite[p. 341]{Levin}). We now re-cast part of this theorem in the language of zero-sector preservers.  

\begin{proposition}\label{RoMS}
Let $T$ be the linear operator on $\R[z]$ defined by $T[z^n] = \gamma_n z^n$, where $\{\gamma_n\}_{n=0}^{\infty}$ is a sequence of real numbers. If $T$ is a zero-sector preserver for the sector $S(0)$, then the sequence $\{\gamma_n\}_{n=0}^{\infty}$ is a multiplier sequence of the first kind.
\end{proposition}

\begin{proof}
Suppose $T$ is a zero-sector preserver for $S(0)$. Then, for every natural number $n$, the zeros of 
\[
T[(1-z)^n] = \sum_{k=0}^{n} \binom{n}{k} \gamma_k (-z)^k
\] 
are all real and positive. From this, we conclude that elements of the sequence $\{\gamma_n\}_{n=0}^{\infty}$ must all be the same sign and that $T[(1+z)^n]$ has only real zeros for all $n$. Therefore, the result follows from the algebraic characterization of multiplier sequences of the first kind.
\end{proof}

Combining the preceding result with Theorem \ref{ZSRO}, we obtain a new proof of another classical result due to Laguerre.

\begin{corollary} \emph{(C.f.,~\cite{L})}
If $q$ satisfies $-1<q<1$, then the linear operator $T \colon \R[z]\to\R[z]$ defined by $T[z^n] = q^{n^2} z^n$ is a zero-sector reducing operator. Thus, the sequence $\{q^{k^2}\}_{k=0}^{\infty}$ is a multiplier sequence of the first kind. 
\end{corollary}

\begin{proof}
Apply Theorem \ref{ZSRO} (with $-\alpha^2/2 = \ln q$ and $\theta=0$) and Proposition \ref{RoMS}.
\end{proof}

%%%%%%%%%%%%%%%%%%%%%%%%%%%%%%%%%%%%%%%%%%%%%%%%%%%%%%%%%%%%%%%%%%%%%%%%%%%%%%%%%%%%%%%%%%%%%%%%%%%%%%%%%%%%%%%%%%%%%%%%%%%%%%%%%%%%%%%%%%%%%%%%%%%%%%%%%%%%%%%%

\subsection{Analysis of Three-Term Polynomials}
Consider the polynomial
\[
p(x) = x^{n+2}- b x^{n+1} + c x^n \qquad(b>0)
\]
with two zeros in the right-half plane and a multiple zero at the origin. Its nonzero zeros are
\[
\frac{b\pm \sqrt{b^2 - 4c}}{2}.
\]
Transforming the polynomial by the linear operator defined by $T[z^k]= \gamma_k z^k$ gives a polynomial with multiple zeros at the origin and two additional zeros at
\[
\frac{\gamma_{n+1} \cdot (b\pm \sqrt{b^2 - 4cr_n})}{2},
\]
where
\[
r_n = \frac{\gamma_n \gamma_{n+2}}{\gamma_{n+1}^2}.
\]
Therefore, if such an operator is to be a zero-sector reducing operator, then it is necessary that $r_n<1$ for all $n$ and 
\begin{equation}\label{eq:limsupcondition}
\limsup r_n <1.
\end{equation}
This immediately rules out several types of multiplier sequences of the first kind, such as geometric sequences $\{r^k\}_{k=0}^{\infty}$, sequences interpolated by polynomials $\{p(k)\}_{k=0}^{\infty}$, and the reciprocal of the factorial $\{1/k!\}_{k=0}^{\infty}$, as candidates for zero-sector reducing operators.

%%%%%%%%%%%%%%%%%%%%%%%%%%%%%%%%%%%%%%%%%%%%%%%%%%%%%%%%%%%%%%%%%%%%%%%%%%%%%%%%%%%%%%%%%%%%%%%%%%%%%%%%%%%%%%%%%%%%%%%%%%%%%%%%%%%%%%%%%%%%%%%%%%%%%%%%%%%%%%%%

\section{Operators Relating to Double Sectors}
In this section, we explore polynomials that have roots in a double sector $\pm S(\theta)$. We define a double zero-sector preserving operator $T$ to be an operator such that whenever a polynomial $p(z)$ has zeros only inside $\pm S(\theta)$ for any $\theta \in [0,\frac{\pi}{2})$, the transformed polynomial $T[p(z)]$ has zeros only inside $\pm S(\theta)$.

\begin{theorem}
There are no double zero-sector reducing operators $T \colon \R[z] \rightarrow \R[z]$ of the form
%such that
$T[z^k] = \gamma_k z^k$ for all $k$. 
\end{theorem}

\begin{proof}
We consider the polynomial
\[
p(z) = 4+z^4.
\] 
The zeros of $p$ are $\pm 1\pm  i$, which are inside the double sector $\pm S(\frac{\pi}{4})$. If $T$ is diagonal with respect to the standard basis, then  $T[p(z)] = 4\gamma_0  + \gamma_4 z^4$. The zeros of $T[p(z)]$ are therefore the 4th roots of $-4\gamma_0/\gamma_4$, 
which cannot lie within a double sector $\pm S(\alpha)$ with $\alpha<\pi/4$. 
\end{proof}

In the introduction, we mentioned that a multiplier sequence of the second kind can be thought of as an operator on $\R[x]$ which maps the sector $S(0)$ to the double sector $\pm S(0)$. This fact is a consequence of a general theorem regarding multiplier sequences of the second kind. 

\begin{theorem} Let $T\colon\mathbb{C}[z] \to \mathbb{C}[z]$ be a linear operator, with $T[z^k] = c\gamma_kz^k$ for all $k \in \mathbb{N}$, where $c \in \mathbb{C}$ and $\gamma_k \in \mathbb{R}$ and fix a value of $\theta$ satisfying $0 < \theta< \frac{\pi}{2}$. Then $T$ sends any polynomial $p$ with zeros only in $S(\theta)$ to a polynomial whose zeros lie in $\pm S(\theta)$ if and only if $\{\gamma_k\}_{k=0}^{\infty}$ is a multiplier sequence of the second kind.
\end{theorem}

\begin{proof}
If $\{\gamma_k\}_{k=0}^{\infty}$ is a multiplier sequence of the second kind, then by ~\cite[Theorem 129]{C}, it sends polynomials in $S(\theta)$ to polynomials with zeros in the corresponding double sector. 

To prove the converse, we will use a strategy developed by M. Chasse (c.f. the proof of ~\cite[Proposition 138]{C}). Assume, by way of contradiction, that $\{\gamma_k\}_{k=0}^{\infty}$ is not a multiplier sequence of the second kind, but $T$ map polynomials with zeros in the fixed sector $S(\theta)$ 
to polynomials with zeros in the corresponding double sector $\pm S(\theta)$. Since $\{\gamma_k\}_{k=0}^{\infty}$ is not a multiplier sequence of the second kind, we know that there exists an $n \in \mathbb{N}$ such that $T[(z+1)^n]$ has a pair of non-real complex conjugate roots (see, for example, \cite[p. 182]{RS}). Note that $T[(z-1)^n]=(-1)^nT[(z+1)^n]\Big|_{z \to -z}$. Hence $T[(z-1)^n]$ also has a pair of non-real complex-conjugate zeros, say $\zeta$ and $\bar{\zeta}$, and consequently, $e^{i \alpha}\zeta$ is a zero of $T[(e^{-i\alpha}z-1)^n]$.

If $\zeta$ lies in the first quadrant, then we can choose $\alpha$ so that $0<\alpha<\theta$ and $\theta<\alpha+\arg(\zeta)<\pi/2$ (see Figure \ref{figure: rotate1}). If $\zeta$ lies in the third quadrant, then we can choose $\alpha$ so that $0<\alpha<\theta$ and $\theta-\pi<\alpha+\arg(\zeta)<-\pi/2$ (see Figure \ref{figure: rotate2}). 
\begin{figure}[hbt!] 
\caption{$\zeta$ in the first quadrant.}
\label{figure: rotate1}
\includegraphics[width=2 in]{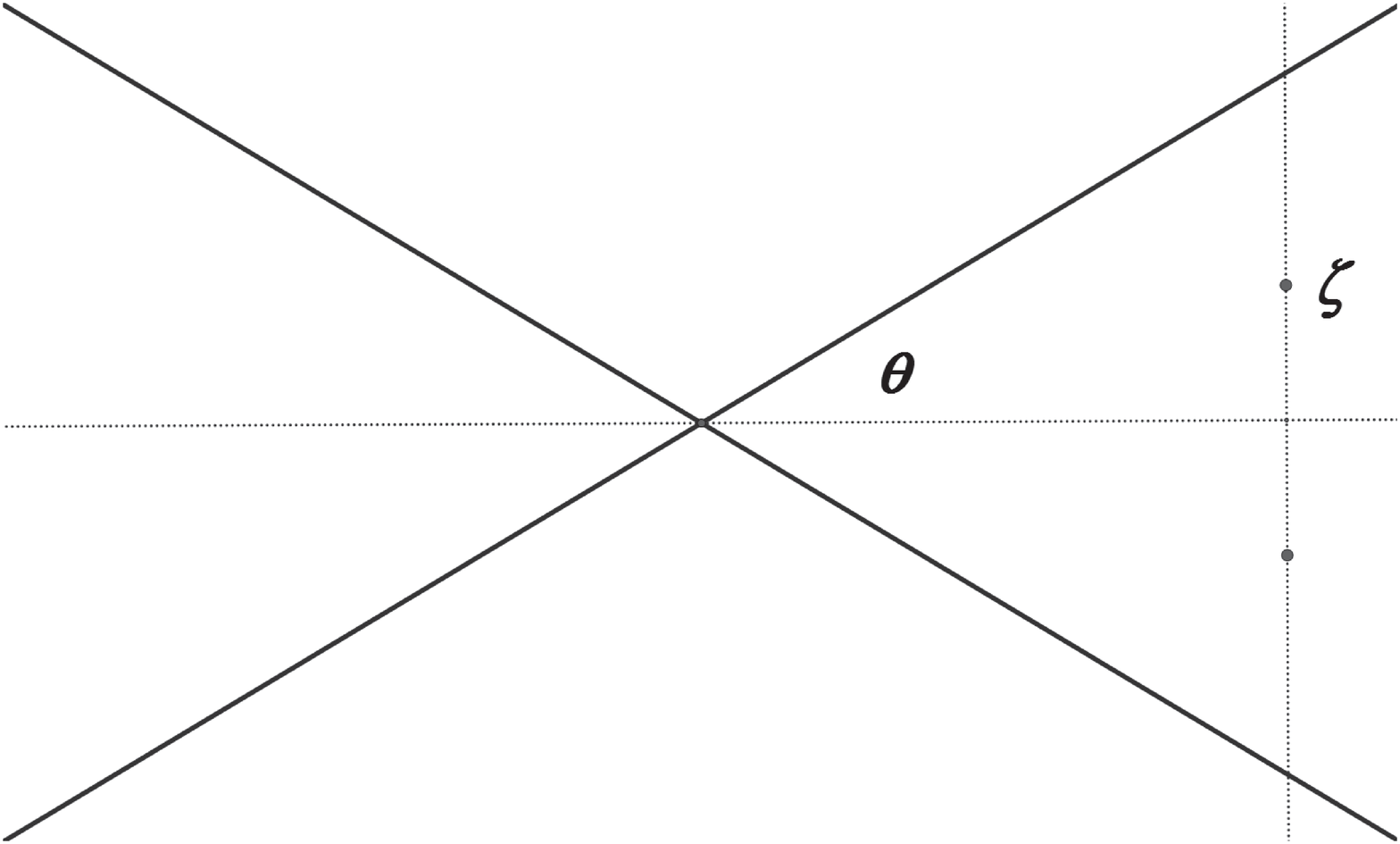} \hskip 1 in \includegraphics[width=2 in]{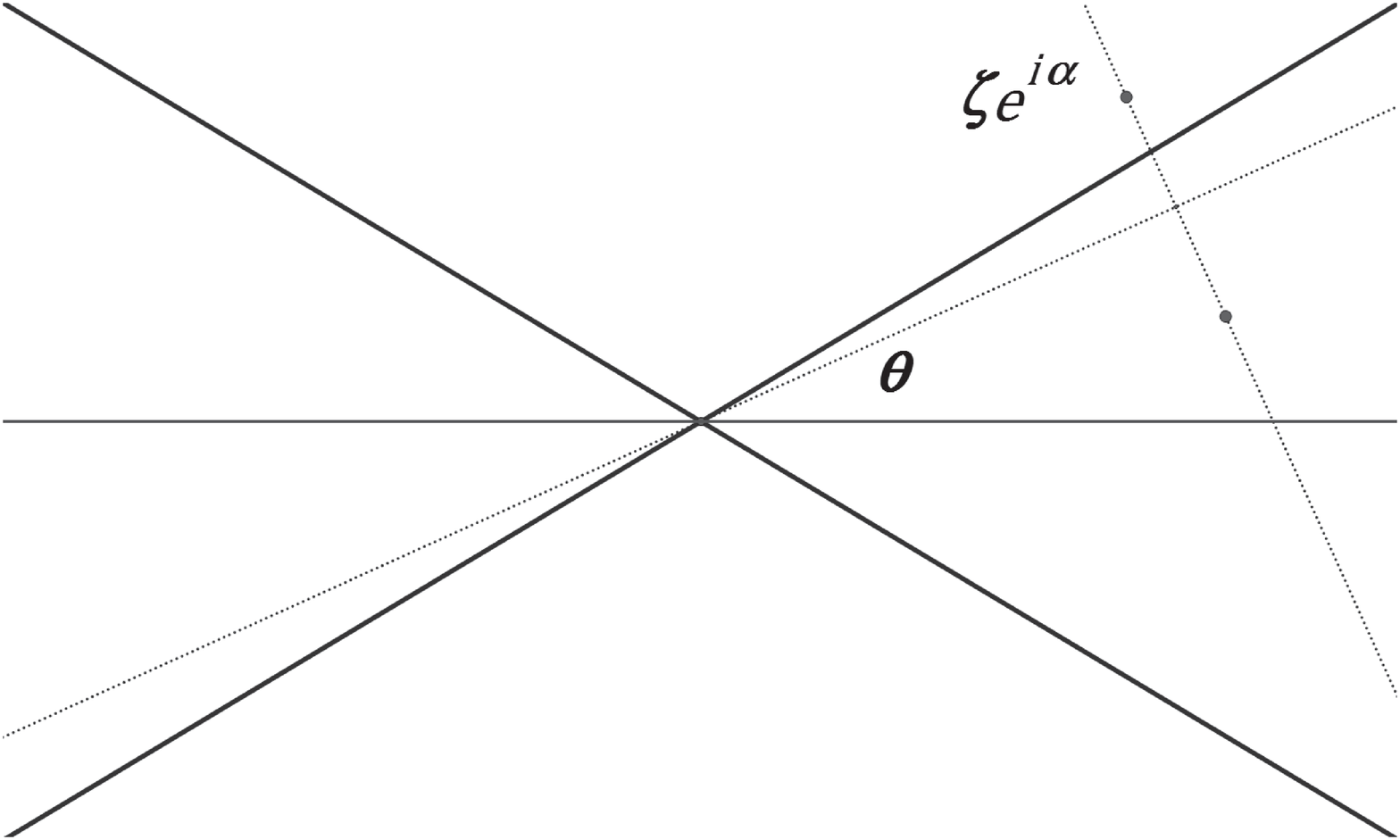}
\end{figure}
\begin{figure}[hbt!] 
\caption{$\zeta$ in the third quadrant.}
\label{figure: rotate2}
\includegraphics[width=2 in]{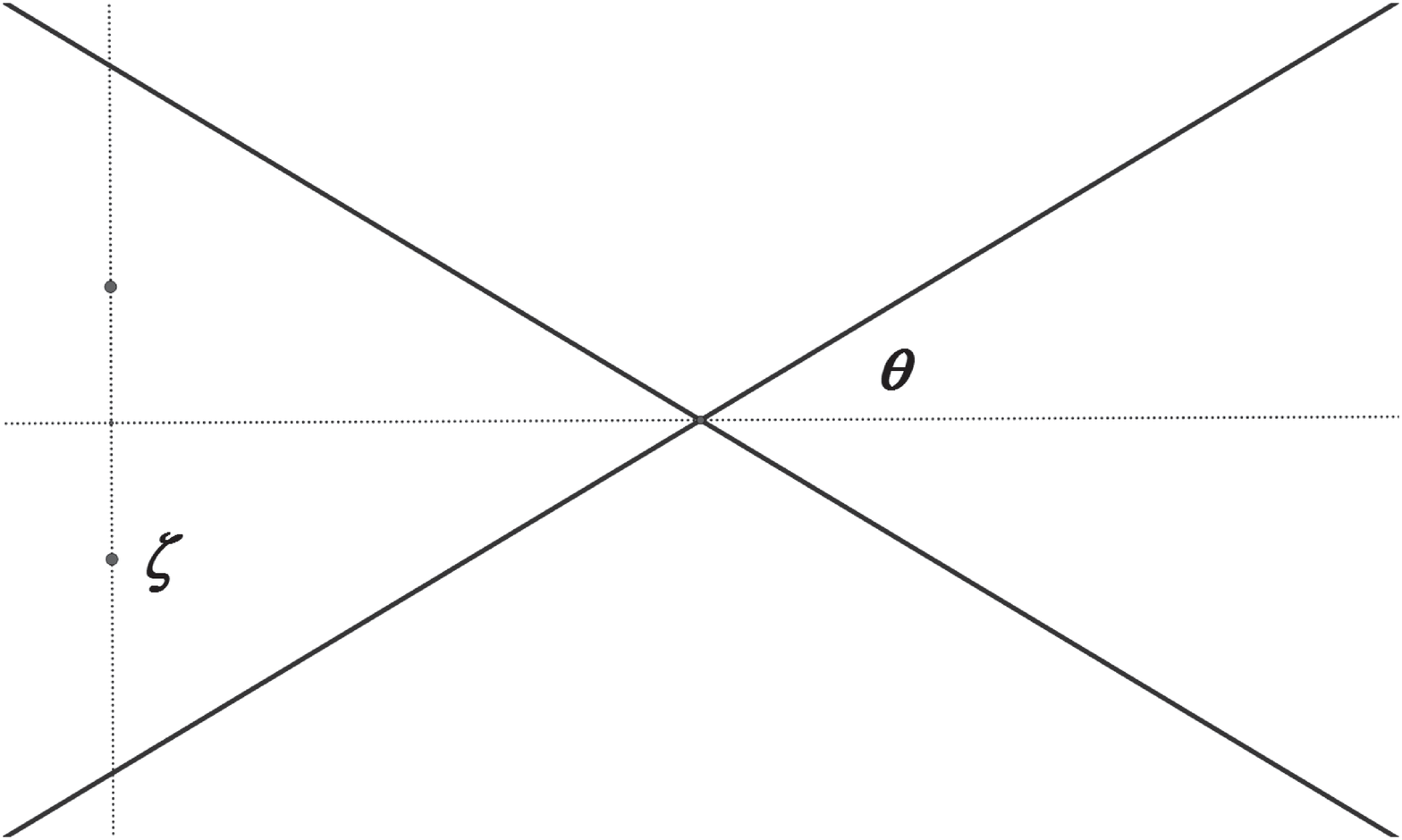} \hskip 1 in \includegraphics[width=2 in]{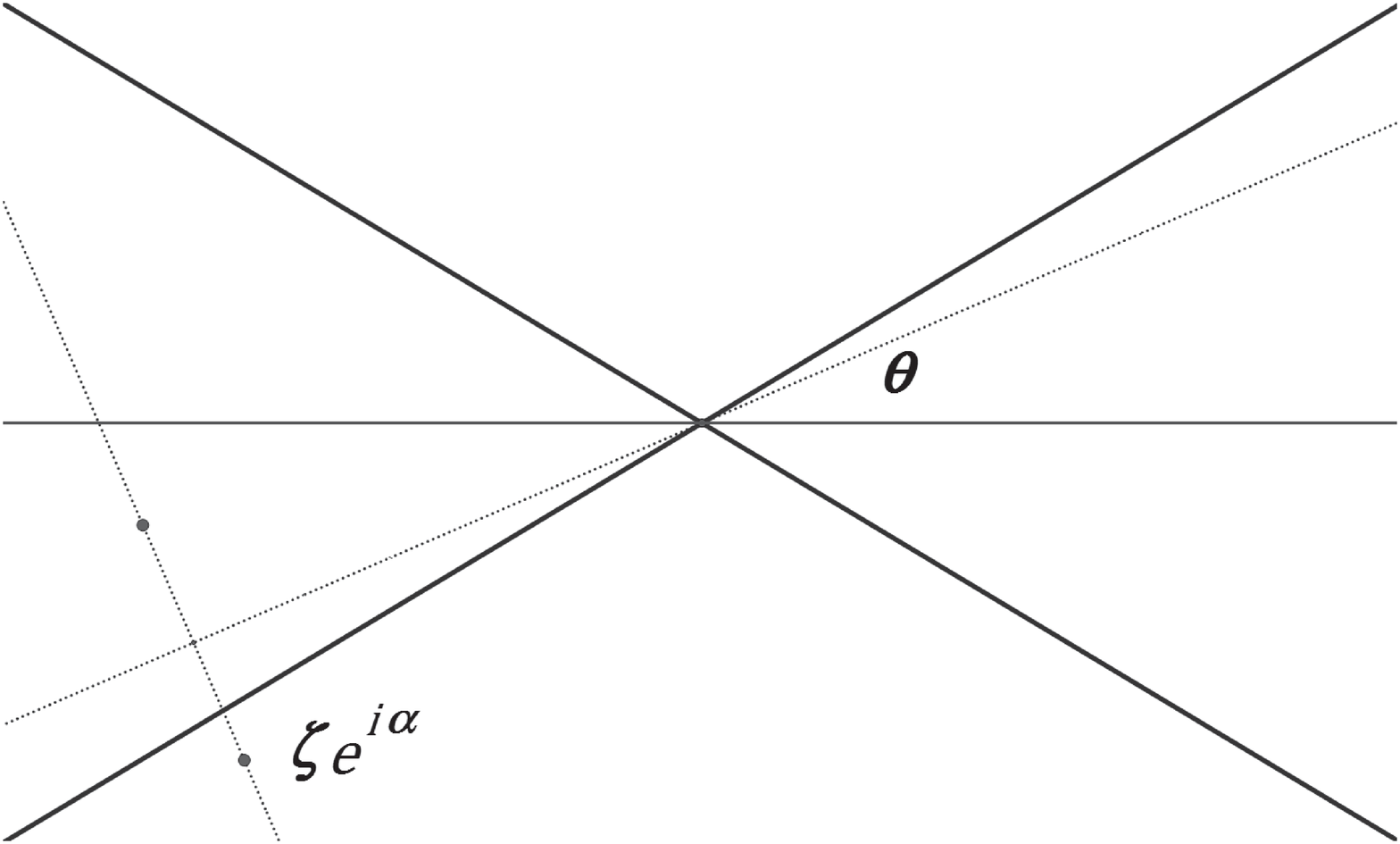}
\end{figure}
In either case, the polynomial $(e^{-i\alpha}z-1)^n$, with zeros in $S(\theta)$ is mapped to a polynomial which has a zero outside $\pm S(\theta)$, reaching a contradiction as desired.
\end{proof}

One interesting note about this theorem is that if $T$ sends polynomials with roots in one sector to polynomials with roots in its corresponding double sector, then it must be a multiplier sequence of the second kind, and therefore it sends polynomials with zeros in $S(\theta)$ to polynomials with zeros in the corresponding double sector $\pm S(\theta)$ \textit{for all choices of $\theta$}. Also worth noting is that this theorem holds without the assumption that our sector is centered about the positive real line.

%%%%%%%%%%%%%%%%%%%%%%%%%%%%%%%%%%%%%%%%%%%%%%%%%%%%%%%%%%%%%%%%%%%%%%%%%%%%%%%%%%%%%%%%%%%%%%%%%%%%%%%%%%%%%%%%%%%%%%%%%%%%%%%%%%%%%%%%%%%%%%%%%%%%%%%%%%%%%%%%

\section{Open Problems}

In this final section, we detail some open problems arising from our work. Guided by the necessary conditions determined in this paper and the success with the exponential sequence, we pose the following question.
\begin{question}
Let $\alpha$ be a positive number. Is it true that the linear operator $T$ on $\R[x]$ defined by $T[z^k]=\exp(-\alpha k^p)z^k$ for all $k$ is a zero-sector reducing operator if and only if $p\geq 2$?
\end{question}
We remark that the necessary condition stipulated in equation (\ref{eq:limsupcondition}) is not satisfied when $p<2$. 

More generally, we can ask the following. 
\begin{question}
Can one determine a complete characterization of the zero-sector reducing operators on $\R[x]$?
\end{question}
We note that any composition of zero-sector reducing and zero-sector preserving operators will result in a zero-sector reducing operator. For example, the sequence $\{(1+k)\exp(k-k^2)\}$ gives rise to a zero-sector reducing operator, since it is a composition of operators corresponding to the positive multiplier sequences of the first kind $\{1+k\}$ and $\{\exp(k)\}$, and the zero-sector reducing operator $\{\exp(-k^2)\}$.

It would be interesting to obtain examples of zero-sector reducing operators which were not diagonal with respect to the standard basis and/or were not linear operators. To date, none are known.

%%%%%%%%%%%%%%%%%%%%%%%%%%%%%%%%%%%%%%%%%%%%%%%%%%%%%%%%%%%%%%%%%%%%%%%%%%%%%%%%%%%%%%%%%%%%%%%%%%%%%%%%%%%%%%%%%%%%%%%%%%%%%%%%%%%%%%%%%%%%%%%%%%%%%%%%%%%%%%%%

\end{document}